\documentclass[reqno]{amsart}
\usepackage{amsmath}
\usepackage{amssymb}
\usepackage{amsthm}
\usepackage[arrow,matrix]{xy}
\begin{document}
\def\eq#1{{\rm(\ref{#1})}}
\theoremstyle{plain}
\newtheorem{thm}{Theorem}[section]
\newtheorem{lem}[thm]{Lemma}
\newtheorem{prop}[thm]{Proposition}
\newtheorem{cor}[thm]{Corollary}
\theoremstyle{definition}
\newtheorem{dfn}[thm]{Definition}
\newtheorem*{rem}{Remark}
\def\coker{\mathop{\rm coker}}
\def\ind{\mathop{\rm ind}}
\def\Re{\mathop{\rm Re}}
\def\vol{\mathop{\rm vol}}
\def\SO{\mathbin{\rm SO}}
\def\Im{\mathop{\rm Im}}
\def\im{\mathop{\rm im}}
\def\min{\mathop{\rm min}}
\def\Spec{\mathop{\rm Spec}\nolimits}
\def\Hol{{\textstyle\mathop{\rm Hol}}}
\def\Ind{\mathop{\rm Ind}}
\def\ge{\geqslant}
\def\le{\leqslant}
\def\C{{\mathbin{\mathbb C}}}
\def\R{{\mathbin{\mathbb R}}}
\def\N{{\mathbin{\mathbb N}}}
\def\Z{{\mathbin{\mathbb Z}}}
\def\D{{\mathbin{\mathcal D}}}
\def\H{{\mathbin{\mathcal H}}}
\def\M{{\mathbin{\mathcal M}}}
\def\U{{\mathbin{\mathcal U}}}
\def\al{\alpha}
\def\be{\beta}
\def\ga{\gamma}
\def\de{\delta}
\def\ep{\epsilon}
\def\io{\iota}
\def\ka{\kappa}
\def\la{\lambda}
\def\ze{\zeta}
\def\th{\theta}
\def\vt{\vartheta}
\def\vp{\varphi}
\def\si{\sigma}
\def\up{\upsilon}
\def\om{\omega}
\def\De{\Delta}
\def\Ga{\Gamma}
\def\Th{\Theta}
\def\La{\Lambda}
\def\Om{\Omega}
\def\Up{\Upsilon}
\def\ts{\textstyle}
\def\sst{\scriptscriptstyle}
\def\sm{\setminus}
\def\na{\nabla}
\def\pd{\partial}
\def\op{\oplus}
\def\ot{\otimes}
\def\bigop{\bigoplus}
\def\iy{\infty}
\def\ra{\rightarrow}
\def\longra{\longrightarrow}
\def\dashra{\dashrightarrow}
\def\t{\times}
\def\w{\wedge}
\def\bigw{\bigwedge}
\def\d{{\rm d}}
\def\bs{\boldsymbol}
\def\ci{\circ}
\def\ti{\tilde}
\def\ov{\overline}
\def\md#1{\vert #1 \vert}
\def\nm#1{\Vert #1 \Vert}
\def\bmd#1{\big\vert #1 \big\vert}
\def\cnm#1#2{\Vert #1 \Vert_{C^{#2}}} 
\def\lnm#1#2{\Vert #1 \Vert_{L^{#2}}} 
\def\bnm#1{\bigl\Vert #1 \bigr\Vert}
\def\bcnm#1#2{\bigl\Vert #1 \bigr\Vert_{C^{#2}}} 
\def\blnm#1#2{\bigl\Vert #1 \bigr\Vert_{L^{#2}}} 
\title[Deformations of AC Special Lagrangian Submanifolds]{Deformations of
Asymptotically Cylindrical Special Lagrangian Submanifolds with
Moving Boundary}

\author[Sema Salur and A. J. Todd]{Sema Salur and A. J. Todd}

\address {Department of Mathematics, University of Rochester, Rochester, NY, 14627}
\email{salur@math.rochester.edu }

\address {Department of Mathematics, University of Rochester,
Rochester, NY, 14627}
\email{ajtodd@math.rochester.edu }

\begin{abstract}
In \cite{SaAJ}, we proved that, under certain hypotheses, the moduli space of an asymptotically cylindrical special Lagrangian submanifold with fixed boundary of an asymptotically cylindrical Calabi-Yau $3$-fold is a smooth manifold. Here we prove the analogous result for an asymptotically cylindrical special Lagrangian submanifold with moving boundary.
\end{abstract}

\date{}
\maketitle
\section{Introduction}

The goal of the present paper is to prove the following theorem:

\begin{thm}
Assume that $(X,\om,g,\Om)$ is an asymptotically cylindrical Calabi-Yau $3$-fold, asymptotic to $M\times S^1\times(R,\infty)$ with decay rate $\al<0$, where $M$ is a compact, connected $K3$ surface, and that $L$ is an asymptotically cylindrical special Lagrangian $3$-submanifold in $X$ asymptotic to $N\times\{p\}\times (R',\infty)$ for $R'>R$ with decay rate $\be$ for $\al\leq\be<0$ where $N$ is a compact special Lagrangian $2$-fold in $M$ and $p\in S^1$. Let $\ga<0$ be such that $\be<\ga$ and such that $(0,\ga^2]$ contains no eigenvalues of the Laplacian on functions on $N$ nor eigenvalues of the Laplacian on exact $1$-forms on $N$.

Let $\Up:H^1(N,\R)\to H^2_{cs}(L,\R)$ be the linear map coming from the long exact sequence in cohomology (so that $\ker\Up$ is a vector subspace of $H^1(N)$), let $\U$ be a small open neighborhood of $0$ in $\ker\Up$ and let $N_s$ denote the special Lagrangian submanifolds of $M$ near $N$ for some $s\in\U$. Then the moduli space $M^{\ga}_L$ of asymptotically cylindrical special Lagrangian submanifolds in $X$ close to $L$, and asymptotic to $N_s\t\{p\}\t(R',\infty)$ with decay rate $\ga$ is a smooth manifold of dimension $\dim V+\dim\ker\Up=\dim V+b^2(L)-b^1(L)+b^0(L)-b^2(N)+b^1(N)$, where $V$ is the image of $H^1_{cs}(L)$ in $H^1(L)$.
\label{mslthm1}
\end{thm}

We begin the paper with a recapitulation of the definitions from Calabi-Yau and special Lagrangian geometry and the definitions of asymptotically cylindrical Calabi-Yau $3$-folds and their asymptotically cylindrical special Lagrangian submanifolds. Then we recall a bit of the theory of elliptic operators on asymptotically cylindrical manifolds developed by Lockhart and McOwen in \cite{LoMc,Lock}. Next we provide a sketch of the proof of the fixed boundary case from \cite{SaAJ} (in fact, portions of that proof will carry over directly to our current work), followed by the proof of Theorem \ref{mslthm1}. Finally, as a application of our work in these papers, we prove the following theorem of Leung:

\begin{thm}
Let $X$ be a Calabi-Yau $3$-fold asymptotic to $M\t S^1\t (R,\infty)$, $M$ a $K3$-surface, and let $L$ be a special Lagrangian submanifold of $X$ asymptotic to $N\t\{p\}\t (R',\infty)$, $N$ a special Lagrangian submanifold of $M$. Let $\mathcal{M}^{SLag}(X)$ be the moduli space of special Lagrangian \emph{cycles} in $X$ and $\mathcal{M}^{SLag}(M)$ the moduli space of special Lagrangian \emph{cycles} in $M$. Then the boundary map $b:\mathcal{M}^{SLag}(X)\to\mathcal{M}^{SLag}(M)$ given by $b(L,D_E)=(N,D_{E'})$ is a Lagrangian immersion.
\label{lagimm}
\end{thm}

\section{Calabi-Yau and Special Lagrangian Geometry}

We begin with the definitions of a Calabi-Yau $n$-fold and special Lagrangian $n$-submanifold, and we give the definitions of (asymptotically) cylindrical Calabi-Yau $3$-folds and (asymptotically) cylindrical special Lagrangian $3$-submanifolds. References for this section include: Harvey and Lawson, \cite{HaLa}; Joyce, \cite{Joyc1}; McLean, \cite{McLe}; and Kovalev, \cite{Kova}.

\begin{dfn}
A \emph{complex} $n$-dimensional {\em Calabi-Yau} manifold $(X,\om,g,\Om)$ is a K\"ahler manifold with zero first Chern class, that is $c_1(X)=0$. In this case, $(X,\om,g,\Om)$ is also called a \emph{Calabi-Yau $n$-fold}.
\end{dfn}

\begin{rem}
As the complex structure coming from the K\"ahler condition will play no direct role in our work, we omit it from our notation.
\end{rem}

\begin{dfn}
A \emph{real} $n$-dimensional submanifold $L\subseteq X$ is \emph{special Lagrangian} if $L$ is Lagrangian (i.e. $\om|_L\equiv 0$) and $\Im\Om$ restricted to $L$ is zero. In this case, we will also call $L$ a \emph{special Lagrangian $n$-submanifold}.
\end{dfn}

\begin{rem}
$\Re\Om$ restricts to be the volume form with respect to the induced metric on special Lagrangian submanifolds, so that the special Lagrangian submanifolds $L$ of a Calabi-Yau manifold $X$ are exactly the calibrated submanifolds of $X$ with respect to the calibration form $\Re\Om$ (see \cite[Section 3.1]{HaLa} for more information regarding this point).
\end{rem}

We now define cylindrical and asymptotically cylindrical Calabi-Yau manifolds. See \cite{Kova} for more information on the definitions in this section.

\begin{dfn}
A Calabi-Yau $3$-fold $(X_0,\om_0,g_0,\Om_0)$ is called {\em cylindrical} if $X_0=M\times S^1\times\R$ where $M$ is a connected, compact $K3$ surface with K\"ahler form $\ka_I$ and holomorphic $(2,0)$-form $\ka_J+i\ka_K$, and $(\om_0,g_0,\Om_0)$ is compatible with the product structure $M\times S^1\times\R$, that is, $\om_0=\ka_I+\d\th\w\d t$, $\Om_0=(\ka_J+i\ka_K)\w(\d\th+i\d t)$ and $g_0=g_{(M\times S^1)}+\d t^2$.
\label{acdfn1}
\end{dfn}

\begin{rem}
The indices on the $2$-forms $\ka_I$, $\ka_J$ and $\ka_K$ on $M$ are meant to reflect the hyperk\"ahler structure of $M$. (This is consistent with the notation in \cite{Kova}.)
\end{rem}

\begin{dfn}
A connected, complete Calabi-Yau $3$-fold $(X,\om,g,\Om)$ is called {\em asymptotically cylindrical with decay rate} $\al<0$ if there exists a cylindrical Calabi-Yau $3$-fold $(X_0,\om_0,g_0,\Om_0)$ as in Definition \ref{acdfn1}, a compact subset $K\subset X$, a real number $R>0$, and a diffeomorphism $\Psi: M\times S^1\times (R,\infty) \rightarrow X\setminus K$ such that $\Psi^*(\om)=\om_0 +\d\xi_1$ for some $1$-form $\xi_1$ with $\bmd{\na^k\xi_1}=O(e^{\al t})$ and $\Psi^*(\Om) =\Om_0 +\d\xi_2$ for some complex $2$-form $\xi_2$ with $\bmd{\na^k\xi_2}=O(e^{\al t})$ on $M\times S^1\times\R$ for all $k\ge 0$, where $\na$ is the Levi-Civit\`a connection of the cylindrical metric $g_0=g_{(M\times S^1)} + dt^2$.
\label{acdfn2}
\end{dfn}

This definition implies that the restriction of the K\"ahler form $\om$ and holomorphic $(3,0)$-form $\Om$ equals to the above for large $t$, up to a possible error of order $O(e^{\al t})$ for some parameter $\al<0$. Notice also in this definition that we assume $X$ and $M$ are connected, so that $X$ only has one end; since we are working with Ricci-flat manifolds, this is not a restrictive assumption by a result obtained by the first author in \cite{Salu1}.

We now define cylindrical and asymptotically
cylindrical special Lagrangian submanifolds similarly.

\begin{dfn}
Let $(X_0,\om_0,g_0,\Om_0)$ be a cylindrical Calabi-Yau $3$-fold as in Definition \ref{acdfn1}. A $3$-dimensional submanifold $L_0$ of $X_0$ is called {\em cylindrical special Lagrangian} if $L_0=N\times \{p\}\times\R$ for some compact special Lagrangian submanifold $N$ in $M$, a point $p$ in $S^1$ and the restrictions of $\om_0$ and $\Im\Om_0$ to $L_0$ are zero, that is, $\om_0|_{L_0}=(\kappa_I+\d\th\w\d t)|_{L_0} =0$ and $\Im(\Om_0)|_{L_0}=\Im((\kappa_J+i\kappa_K)\w(\d t+i\d\th))|_{L_0}=0$.
\label{acdfn3}
\end{dfn}

\begin{dfn}
Let $(X_0,\om_0,g_0,\Om_0)$ , $M\times S^1\times\R$, $(X,\om,g,\Om)$, $K, R, \Psi$ and $\al$ be as in Definitions \ref{acdfn1} and \ref{acdfn2}, and let $L_0=N\times \{p\}\times\R$ be a cylindrical special Lagrangian $3$-fold in $X_0$ as in Definition \ref{acdfn3}.

A connected, complete special Lagrangian 3-fold $L$ in
$(X,\om,g,\Om)$ is called \emph{asymptotically cylindrical with decay rate} $\be$ for $\al\leq\be<0$ if there exists a compact
subset $K'\subset L$, a normal vector field $v$ on
$N\t\{p\}\t(R',\infty)$ for some $R'>R$, and a diffeomorphism
$\Phi:N\t\{p\}\t(R',\infty)\ra L\sm K'$ such that the following
diagram commutes:
\begin{equation}
\begin{gathered}
\xymatrix{M\t S^1\t (R',\iy) \ar[d]^\subset & N\t\{p\}\t (R',\iy)
\ar[l]^{\;\;\;\;\exp_v} \ar[r]_{\;\;\;\;\;\;\;\;\;\Phi} & (L\sm
K')
\ar[d]^\subset \\
M\t S^1\t (R,\iy)\ar[rr]^\Psi && (X\sm K), }
\end{gathered}
\label{aceq1}
\end{equation}
and $\bmd{\na^kv}=O(e^{\be t})$ on $N\t\{p\}\t(R',\iy)$ for
all $k\geq 0$.
\label{acdfn4}
\end{dfn}

Notice that here again, we require $L$ to be connected; however, unlike the case of the ambient manifold, it may be that $N$ is \emph{not} connected, so that $L$ may have multiple ends. This will not present a problem in the analysis since we are assuming the ends have the same decay rate.

\section{Analysis on Asymptotically Cylindrical Special Lagrangian $3$-Manifolds}

Lockhart and McOwen \cite{Lock, LoMc} setup a general framework for elliptic operators on noncompact manifolds where the basic tools are weighted Sobolev spaces. Using these spaces and the notion of asymptotically cylindrical linear elliptic partial differential operators, they get weighted Sobolev embedding theorems and recover elliptic regularity and Fredholm results.
\begin{rem}
For an alternate approach to these types of results on noncompact manifolds, see Melrose \cite{Mel1, Mel2}.
\end{rem}

Let $L$ be an asymptotically cylindrical special Lagrangian $3$-manifold with data as in Definition \ref{acdfn4}.

\begin{dfn}
Let $A$ be a vector bundle on $L$ with smooth metrics $h$ on the fibers and a connection $\na_A$ on $A$ compatible with $h$; let $A_0$ be a vector bundle on $N\times\{p\}\t\R$ that is invariant under translations in $\R$, that is a \emph{cylindrical} vector bundle, with metrics $h_0$ on the fibers and $\na_{A_0}$ a connection on $A_0$ compatible with $h_0$ where $h_0$ and $\na_{A_0}$ are also invariant under translations in $\R$.

$A$, $h$ and $\na_A$ are said to be \emph{asymptotically cylindrical}, asymptotic to $A_0$, $h_0$ and $\na_{A_0}$ respectively, if $\Phi^*(A)\cong A_0$ on $N\t\{p\}\t(R',\infty)$ with $|\Phi^*(h)-h_0|=O(e^{\be t})$ and $|\Phi^*(\na_A)-\na_{A_0}|=O(e^{\be t})$ as $t\to\infty$.
\end{dfn}

Recall that for $k\geq 0$,  $$L^p_k(A)=\left\{s\in\Ga(A): s\text{ is }k\text{-times weakly differentiable and }\Vert s\Vert_{L^p_k}<\infty\right\},$$
where $\Ga(A)$ is the space of \emph{all} sections of $A$,
$$\Vert s\Vert_{L^p_k}=\left(\sum_{j=0}^k \int_L \left|\na^j_As\right|^p \text{ } dvol_L\right)^{1/p}$$
and $dvol_L$ is the volume element of $L$, and $$L^p_{k,\rm{loc}}(A)=\{s\in\Ga(A):\phi s\in L^p_k(A) \text{ for all }\phi\in C^\infty_0(L)\},$$ where $C^\infty_0(L)$ is the set of all compactly-supported smooth functions on $L$.

\begin{dfn}
Choose a smooth function $\rho:L\to\R$ such that $\Phi^*(\rho)\equiv t$ on $N\times\{p\}\t(R',\infty)$. By Definition \ref{acdfn4}, this condition prescribes $\rho$ on $L\setminus K'$, so it is only necessary to smoothly extend $\rho$ over $K'$. For $p\geq 1$, $k\geq 0$ and $\ga\in\R$, the \emph{weighted Sobolev space} $L^p_{k,\ga}(A)$ is then the set of sections $s$ of $A$ such that $s\in L^p_{\rm{loc}}(A)$, $s$ is $k$-times weakly differentiable and $$\Vert s \Vert_{L^p_{k,\ga}}=\left(\sum_{j=0}^k \int_L e^{-\ga\rho}\left|\na^j_As\right|^p\text{ } dvol_L \right)^{1/p} < \infty.$$
In particular, $L^p_{k,\ga}(A)$ is a Banach space; further, note that different choices of $\rho$ give the same space $L^p_{k,\ga}(A)$ with equivalent norms since $\rho$ is uniquely determined outside of $K$.

Define the Banach space $C^k_{\ga}(A)$ of continuous sections $s$ of $A$ with $k$ continuous derivatives such that $e^{-\ga\rho}|\na^j_As|$ is bounded for each $j=0,1,\ldots,k$ where the norm is given by $$\Vert s \Vert_{C^k_{\ga}}=\sum_{j=0}^k \sup_{L}e^{-\ga\rho}\left|\na^j_As\right|.$$
\label{wss}
\end{dfn}

In general, from \cite[Theorem 1.2]{Bart}, \cite[Theorem 3.10]{Lock} and \cite[Lemma 7.2]{LoMc} there is the following Weighted Sobolev Embedding Theorem (adapted to our case):

\begin{thm}[Weighted Sobolev Embedding Theorem]
Suppose that $k,l$ are integers with $k\geq l\geq 0$ and that $p,q,\ga,\ov{\ga}$ are real numbers with $p,q>1$. Then
\begin{enumerate}
    \item If $\frac{k-l}{3}\geq\frac{1}{p}-\frac{1}{q}$ and $\ov{\ga}\geq\ga$ there is a continuous embedding of $L^p_{k,\ga}(A)\hookrightarrow L^q_{l,\ov{\ga}}(A)$;
    \item If $\frac{k-l}{3}>\frac{1}{p}$ and $\ov{\ga}\geq\ga$ there is a continuous embedding of $L^p_{k,\ga}(A)\hookrightarrow C^l_{\ov{\ga}}(A)$;
\end{enumerate}
\label{embedding}
\end{thm}

Now suppose that $A,B$ are two asymptotically cylindrical vector bundles on $L$ which are asymptotic to the cylindrical vector bundles $A_0,B_0$ on $N\times\{p\}\t\R$. Let $F_0:C^{\infty}(A_0)\to C^{\infty}(B_0)$ be a cylindrical linear partial differential operator of order $k$, that is, a linear partial differential operator which is invariant under translations in $\R$, from \emph{smooth} sections $C^{\infty}(A_0)$ of $A_0$ to \emph{smooth} sections $C^{\infty}(B_0)$ of $B_0$. Let $F:C^{\infty}(A)\to C^{\infty}(B)$ be a linear partial differential operator of order $k$ on $L$.

\begin{dfn}
$F$ is said to be an \emph{asymptotically cylindrical operator}, asymptotic to $F_0$, if, under the identifications $\Phi^*(A)\cong A_0$, $\Phi^*(B)\cong B_0$ on $N\times\{p\}\t\R$, $|\Phi^*(F)-F_0|=O(e^{\be t})$ as $t\to\infty$.
\end{dfn}

From these definitions $F$ extends to a bounded linear operator
$$F^p_{k+l,\ga}:L^p_{k+l,\ga}(A)\to L^p_{l,\ga}(B)$$
for all $p>1$, $l\geq 0$ and $\ga\in\R$. We then have the following elliptic regularity result \cite[Theorem 3.7.2]{Lock}:

\begin{thm}
Let $F$ and $F_0$ be as above. If $1<p<\infty$, $l\in\Z$ and $\ga\in\R$, then for all $s\in L^p_{k+l,loc}(A)$,
$$\Vert s\Vert_{L^p_{k+l,\ga}}\leq C(\Vert Fs\Vert_{L^p_{l,\ga}}+\Vert s\Vert_{L^p_{l,\ga}})$$
for some $C$ independent of $s$.
\label{reg}
\end{thm}

\begin{dfn}
Assume now that $F,F_0$ are also elliptic on $L$ and $N\times\{p\}\t\R$ respectively. Extend $F_0$ to $F_0:C^{\infty}(A_0\ot_{\R}\C)\to C^{\infty}(B_0\ot_{\R}\C)$, and define $\D_{F_0}$ as the set of $\ga\in\R$ such that for some $\de\in\R$ there exists a nonzero section $s\in C^{\infty}(A_0\ot_{\R}\C)$, invariant under translations in $\R$, such that $F_0(e^{(\ga+i\de)t}s)=0$.
\end{dfn}

\noindent The conditions for $F^p_{k+l,\ga}$ to be Fredholm are now given by \cite[Theorem 1.1]{LoMc}.

\begin{thm}
Using the setup of this section, $\D_{F_0}$ is a discrete subset of $\R$, and for $p>1$, $l\geq 0$ and $\ga\in\R$, $F^p_{k+l,\ga}:L^p_{k+l,\ga}(A)\to L^p_{l,\ga}(B)$ is Fredholm if and only if $\ga\not\in\D_{F_0}$.
\label{fred}
\end{thm}

This theorem and Theorem \ref{reg} imply that $\ker(F^p_{k+l,\ga})$ is a finite-dimensional vector space of smooth sections of $A$ whenever $\ga\not\in\D_{F_0}$, and from the Weighted Sobolev Embedding Theorem and the fact that $\ker(F^p_{k+l,\ga})$ is invariant under small changes in $\ga$ we have the following:

\begin{lem}
If $\ga\not\in\D_{F_0}$, then the kernel $\ker(F^p_{k+l,\ga})$ is the same for any choices of $p>1$ and $l\geq 0$ and is a finite-dimensional vector space consisting of smooth sections of $A$.
\label{ker}
\end{lem}

Let $F^*:C^{\infty}(B)\to C^{\infty}(A)$ be the formal adjoint of $F$; that is, $F^*$ is the asymptotically cylindrical linear elliptic partial differential operator of order $k$ on $L$ such that
$$\langle Fs,\tilde{s}\rangle_{L^2(B)}=\langle s,F^*\tilde{s}\rangle_{L^2(A)}$$
for all compactly-supported $s\in C^{\infty}(A)$, $\tilde{s}\in C^{\infty}(B)$.

Then for $p>1$, $l\geq 0$ and $\ga\not\in\D_{F_0}$, $(F^*)^q_{-l,-\ga}:L^q_{-l,-\ga}(B)\to L^q_{-k-l,-\ga}(A)$ is the dual operator of $F^p_{k+l,\ga}$ where $q>1$ satisfies $\frac{1}{p}+\frac{1}{q}=1$, $L^q_{-l,-\ga}(B)$ and $L^q_{-k-l,-\ga}(A)$ are isomorphic as Banach spaces to the dual spaces $(L^p_{l,\ga}(B))^*$ and $(L^p_{k+l,\ga}(A))^*$ respectively and these isomorphisms identify $(F^*)^q_{-l,-\ga}$ with $(F^p_{k+l,\ga})^*$. Further, since Theorem \ref{reg} also holds for negative differentiability, $\ker((F^*)^q_{-l,-\ga})=\ker((F^*)^q_{k+m,-\ga})$ for all $m\in\Z$. This allows us to identify $\coker(F^p_{k+l,\ga})$ with $\ker((F^*)^q_{k+m,-\ga})^*$ for $\ga\not\in\D_{F_0}$, $p,q>1$ with $\frac{1}{p}+\frac{1}{q}=1$ and $l,m\geq 0$; moreover, the index of $F^p_{k+l,\ga}$ is then given by
\begin{equation}
\ind(F^p_{k+l,\ga})=\dim\ker(F^p_{k+l,\ga})-\dim\ker((F^*)^q_{k+m,-\ga}).
\label{ind1}
\end{equation}

\section{Deformations of Asymptotically Cylindrical Special Lagrangian $3$-folds with Fixed Boundary}

In \cite{SaAJ}, we proved the following theorem:

\begin{thm} Assume that $(X,\om,g,\Om)$ is an asymptotically cylindrical Calabi-Yau $3$-fold, asymptotic to $M\times S^1\times(R,\infty)$ with decay rate $\al<0$, where $M$ is a compact, connected $K3$ surface, and that $L$ is an asymptotically cylindrical special Lagrangian $3$-submanifold in $X$ asymptotic to $N\times\{p\}\times (R',\infty)$ for $R'>R$ with decay rate $\be$ for $\al\leq\be<0$ where $N$ is a compact special Lagrangian $2$-fold in $M$ and $p\in S^1$.

If $\ga<0$ is such that $\be<\ga$ and such that $(0,\ga^2]$ contains no eigenvalues of the Laplacian on functions on $N$ nor eigenvalues of the Laplacian on exact $1$-forms on $N$, then the moduli space $\M_L^\ga$ of asymptotically cylindrical special Lagrangian submanifolds in $X$ near $L$, and asymptotic to $N\times\{p\}\times(R',\infty)$ with decay rate $\ga$, is a smooth manifold of dimension $\dim V$ where $V$ is the image of the natural inclusion of $H^1_{cs}(L,\R)$ in $H^1(L,\R)$.
\label{sl1thm}
\end{thm}

\begin{proof}[Sketch of Proof]
Let $(X,\om,g,\Om)$ be an asymptotically cylindrical Calabi-Yau $3$-fold with decay rate $\al<0$, which is asymptotic to the cylindrical Calabi-Yau $3$-fold $(X_0,\om_0,g_0,\Om_0)$. Then there exists a $K3$ surface $M$ such that $X_0=M\times S^1\times\R$ as in Definition \ref{acdfn1}, and there exist $K\subset X$ a compact subset, $R>0$ and a diffeomorphism $\Psi:M\times S^1\times(R,\infty)\to X\setminus K$ such that $|\Psi^*(g)-g_0|=O(e^{\al t})$, $\Psi^*(\om)=\om_0+\d\xi_1$ for some for some $1$-form $\xi_1$ such that $|\na^k\xi_1|=O(e^{\al t})$ for all $k\geq 0$ and $\Psi^*(\Om)=\Om_0+\d\xi_2$ for some holomorphic $2$-form $\xi_2$ such that $|\na^k\xi_2|=O(e^{\al t})$ for all $k\geq 0$.

Let $L$ be an asymptotically cylindrical special Lagrangian $3$-submanifold of $X$ with decay rate $\be$ such that $\al\leq\be<0$, which is asymptotic to the cylindrical special Lagrangian $3$-submanifold $L_0$ of $X_0$. Then there exists a special Lagrangian $2$-submanifold $N$ of $M$ and a point $p\in S^1$ such that $L_0=N\t\{p\}\t\R$, and there exist $K'\subset L$, $R'>R$, a normal vector field $v$ on $N\t\{p\}\t(R',\infty)$ such that $|\na^k v|=O(e^{\be t})$ and a diffeomorphism $\Phi:N\t\{p\}\t(R',\infty)\to L\setminus K'$ making Diagram (\ref{aceq1}) commute.

From the normal bundle $\nu_N$ of $N$ and associated exponential map $\exp_N:\nu_N\to M$ we get the normal bundle $\nu_N\t\{p\}\t(R',\infty)$ of $N\t\{p\}\t(R',\infty)$ with exponential map given by $\exp_N\t\io\t id:\nu_N\t\{p\}\t(R',\infty)\to M\t S^1\t(R',\infty)$ where $\io$ is the natural inclusion map. Then for some $\ep>0$ small enough, we get a diffeomorphism $\exp_N:B_{2\ep}(\nu_N)\to T_N$ of a subbundle $B_{2\ep}(\nu_N)$ of the normal bundle $\nu_N$, whose fibers are that of the ball of radius $2\ep$ about $0$ in the vector space $\nu_N$, with a tubular neighborhood $T_N$ of $N$ in $M$. This in turn gives us a diffeomorphism $\exp_N\t\io\t id:B_{2\ep}(\nu_N)\t\{p\}\t(R',\infty)\to T_N\t S^1\t(R',\infty)$.

Using these maps we are able to construct an identification $\Th:B_{\ep'}(\nu_L)\to T_L$ of a subbundle $B_{\ep'}(\nu_L)$ of the normal bundle $\nu_L$ of $L$ with a tubular neighborhood of $L$ in $X$ in such a way that submanifolds $\tilde{L}$ of $X$ near $L$ correspond to small sections of the normal bundle $\nu_L$ and the asymptotic convergence of $\tilde{L}$ to $N\t\{p\}\t\R$ is reflected in the convergence of these sections to $0$. Finally, we can think of $\Th:B_{\ep'}(T^*L)\to T_L$ since we are on a special Lagrangian manifold.

Let $\eta:L\to T^*L$ be a $1$-form on $L$, and let $\Ga_{\eta}$ denote the graph of $\eta$ in $X$. Define the deformation map $F:B_{\ep'}(T^*L)\to\La^2T^*L\op\La^3T^*L$ by $F(\eta)=((\Th\circ\eta)^*(-\om),(\Th\circ\eta)^*(\Im\Om))$ and extend to the asymptotically cylindrical operator $F=F^p_{2+l,\ga}:L^p_{2+l,\ga}(B_{\ep'}(\nu_L))\to\La^2T^*L\op \La^3T^*L$ where $p>3$, $l\geq 1$ and $\be<\ga<0$. We will say more shortly about $\ga$. Note that $3$-submanifolds $\tilde{L}$ of $X$ are special Lagrangian if and only if $\om|_{\tilde{L}}=0$ and $\Im\Om|_{\tilde{L}}=0$ which happens precisely when $F(\eta)=0$, so that $F^{-1}(0)$ parameterizes the special Lagrangian submanifolds of $X$ near $L$. We then showed that $F:L^p_{2+l,\ga}(B_{\ep'}(\nu_L))\to L^p_{1+l,\ga}(\La^2T^*L)\op L^p_{1+l,\ga}(\La^3T^*L)$ is a smooth map of Banach manifolds whose linearization at $0$ is given by $\d F(0)(\eta)=(\d\eta,*\d^*\eta)=(\d+*\d^*)(\eta)$ where $\d+*\d^*=(\d+*\d^*)^p_{2+l,\ga}:L^p_{2+l,\ga}(B_{\ep'}(\nu_L))\to L^p_{1+l,\ga}(\La^2T^*L)\op L^p_{1+l,\ga}(\La^3T^*L)$ is the asymptotically cylindrical linear elliptic operator.

Now, we want to use the Implicit Mapping Theorem for Banach Manifolds to show that $F^{-1}(0)$ is smooth, finite-dimensional and locally isomorphic to $\ker((\d+*\d^*)^p_{2+l,\ga})$. In particular we need to know when $(\d+*\d^*)^p_{2+l,\ga}$ is Fredholm. It turns out that as long as $\ga<0$ is such that $(0,\ga^2]$ contains neither eigenvalues of the Laplacian $\d^*\d$ on functions on $N$ nor eigenvalues of the Laplacian $\d\d^*$ on exact $1$-forms on $N$, then $(\d+*\d^*)^p_{2+l,\ga}$ will be Fredholm. Further, we are able to determine that $\ker((\d+*\d^*)^p_{2+l,\ga})$ is the vector space of small smooth harmonic $1$-forms on $L$ with dimension equal to the dimension of the image of the natural map $H^1_{cs}(L)\to H^1(L)$ given by $\chi\mapsto[\chi]$ and the cokernel $\coker((\d+*\d^*)^p_{2+l,\ga})\cong(\ker((\d^*+\d*)^q_{2+m,-\ga}))^*$ with $\frac{1}{p}+\frac{1}{q}=1$, $m\geq 1$ where $\ker((\d^*+\d*)^q_{2+m,\ga})$ consists of smooth coclosed $2$-forms and smooth harmonic $3$-forms.

Next, the image of $F$ actually lies in exact $2$- and exact $3$-forms which allows us to say further that the image of $F$ lies in the image of $(\d+*\d^*)^p_{2+l,\ga}$. From here, we apply the Implicit Mapping Theorem for Banach Spaces: define $\mathcal{A}=\ker((\d+*\d^*)^p_{2+l,\ga})$ and $\mathcal{B}$ the subspace of $L^p_{2+l,\ga}(T^*L)$ which is $L^2$-orthogonal to $\mathcal{A}$. Then $\mathcal{A}\op\mathcal{B}=L^p_{2+l,\ga}(T^*L))$ and $(\d+*\d^*)^p_{2+l,\ga}|_{\mathcal{B}}:\mathcal{B}\to\mathcal{C}$ is an isomorphism of vector spaces. Further, it is a homeomorphism of topological spaces by the Open Mapping Theorem. Now, let $\mathcal{U}$, $\mathcal{V}$ be open neighborhoods of $0$ in $\mathcal{A}$, $\mathcal{B}$ respectively such that $\mathcal{U}\t\mathcal{V}\subset L^p_{2+l,\ga}(B_{\ep'}(T^*L))$; then $F:\mathcal{U}\t\mathcal{V}\to\mathcal{C}$ is a smooth map of Banach manifolds, $F(0,0)=(0,0)$ and $\d F(0,0)=(\d+*\d^*)^p_{2+l,\ga}$. Thus, there exists a connected open neighborhood $\mathcal{U'}\subset\mathcal{U}$ of $0$ and a smooth map $G:\mathcal{U'}\to\mathcal{V}$ such that $G(0)=0$ and $F(x)=(x,G(x))\equiv(0,0)$ for all $x\in\mathcal{U'}$, so that near zero $F^{-1}(0,0)=\{(x,G(x)):x\in\mathcal{U'}\}$ is smooth, finite-dimensional and locally isomorphic to $\mathcal{A}=\ker((\d+*\d^*)^p_{2+l,\ga})$.

Finally, we proved that $F^{-1}(0,0)$ consists of smooth sections by proving a regularity result and then that $F^{-1}(0,0)$ is locally homeomorphic to $M^{\ga}_L$, the moduli space of special Lagrangian deformations of $L$, near $L$, in $X$. This completes the proof of the theorem.
\end{proof}

\section{Proof of Theorem \ref{mslthm1}}
Assume that $(X,\om,g,\Om)$, $\al$, $(X_0=M\times S^1\times\R,\om_0,g_0,\Om_0)$, $K\subset X$, $R$, $\Psi$, $L$, $\be$, $L_0=N\t\{p\}\t\R$, $K'\subset L$, $R'$, $v$, $\Phi$, $\nu_N$ and $\Th$ are as in the proof of Theorem \ref{sl1thm}.

We begin with a couple of rigidity results.

\begin{lem}
$N$ is special Lagrangian in $M$ if and only if $N\t\{p\}\t\R$ is special Lagrangian in $M\t S^1\t\R$.
\end{lem}

\begin{proof}
Assume that $N$ is special Lagrangian in $M$. First note that for any $p\in S^1$, $\{p\}\t\R$ is Lagrangian in $S^1\t\R$ since $(\d\th\w\d t)|_{\{p\}\t\R}\equiv 0$. Then $\om_0|_{N\t\{p\}\t\R}=(\kappa_I+\d\th\w\d t)|_{N\t\{p\}\t\R} =\ka_I|_N+(\d\th\w\d t)|_{\{p\}\t\R}\equiv 0$ since $N$ is Lagrangian in $M$ and $\{p\}\t\R$ is Lagrangian in $S^1\t\R$. Next, $\Im\Om_0=\Im[(\ka_J+i\ka_K)\w(\d t+i\d\th)]=\ka_J\w\d\th +\ka_K\w\d t$. Then $(\Im\Om_0)|_{N\t\{p\}\t\R}=(\ka_J\w\d\th)|_{N\t\{p\}} +(\ka_K\w\d t)|_{N\t\R}\equiv 0$. Thus we have that $N\t\{p\}\t\R$ is special Lagrangian in $M\t S^1\t\R$.

Conversely, assume that $N\t\{p\}\t\R$ is special Lagrangian in $M\t\ S^1\t\R$. Then $0\equiv(\ka_I+\d\th\w\d t)|_{N\t\{p\}\t\R}=\ka_I|_N$. Further, $0\equiv(\ka_J\w\d\th+\ka_K\w\d t)|_{N\t\{p\}\t\R}= \ka_K|_N\w\d t|_{\R}$. Since $\d t$ is the volume form on $\R$, we must have $\ka_K|_N\equiv 0$, thus proving that $N$ is special Lagrangian in $M$.
\end{proof}

\begin{lem}
If $L=N\t\{p\}\t[a,b]$ is special Lagrangian in $X=M\t S^1\t[a,b]$ for some $a,b\in\R$, then any special Lagrangian $\tilde{L}$ homologous to $L$ is of the form $\tilde{N}\t\{p\}\t[a,b]$ for some special Lagrangian submanifold $\tilde{N}$ of $M$.
\label{rigidity}
\end{lem}

\begin{proof}
Begin by normalizing the volume of $[a,b]$ to $1$. By the previous result $N$ is special Lagrangian. Let $\tilde{L}$ be any special Lagrangian of $X$ which is homologous to $L$. We begin with the following calculation.
\begin{equation*}
\begin{split}
\vol(\tilde{L})&=\int_{\tilde{L}}\d vol_{\tilde{L}}=\int_{\tilde{L}}\Re(\Om_0)\\
&=\int_{N\t\{p\}\t[a,b]}\Re(\Om_0)=\int_{[a,b]}\vol(N)\d t=\vol(N)\\
\end{split}
\end{equation*}

Now, let $\tilde{L}_t=\tilde{L}\cap(M\t S^1\t\{t\})$ for each $t\in[a,b]$, so that $\vol(\tilde{L}_t)\geq\vol(N)$ since $N$ is a calibrated submanifold of $M$ where equality happens if $\tilde{L}_t$ is also calibrated. Further, we always have $\d vol_{\tilde{L}}\geq\d vol_{\tilde{L}_t}\w\d t$ locally since $\tilde{L}_t\subset\tilde{L}$ which, by Fubini's Theorem, yields that $\vol(\tilde{L})\geq\int_{[a,b]}\vol(\tilde{L}_t)\d t$ where equality occurs precisely when $\tilde{L}$ is a product with $[a,b]$. Finally, we see
\begin{equation*}
\begin{split}
\vol(N)=\vol(\tilde{L})&\geq\int_{[a,b]}\vol(\tilde{L}_t)\d t\\
&\geq\int_{[a,b]}\vol(N)\d t=\vol(N).\\
\end{split}
\end{equation*}
\end{proof}

The next step will be to parameterize the special Lagrangian submanifolds near $N$ in $M$. Recall that $N$ is compact, so the moduli space of special Lagrangian deformations of $N$ near $N$ in $M$ can be identified with the harmonic $1$-forms on $N$ which, again since $N$ is compact, can be identified with $H^1(N,\R)$, the first de Rham cohomology group. This, however, is just a direct sum of $b^1(N)$ copies of $\R$, $b^1(N)$ the first betti number of $N$. For the moment then we can let $\U$ be an open, connected, simply-connected subset of $\R^d$ for some $d\leq b^1(N)$ containing $0$. Now $N=N_0$ can be identified with $0\in\U$, and we can talk about the special Lagrangian submanifolds $N_s$ close to $N_0$ for $s$ close to $0$ in $\U$. (We will be able to give a more concrete description of $\U$ once we have connected all of this to the geometry of $L$.)

Recall that since $N$ is special Lagrangian there is an identification of a tubular neighborhood $T_N$ of $N$ with $1$-forms on $N$. Identifying $N$ with the zero section of $T^*N$, we know that $N_s$ can now be identified with the graph of some harmonic $1$-form $\ze_s$ on $N$. Now, for each $s$, extend $\ze_s$ smoothly to a $1$-form $\si_s$ on $L$ which is asymptotic to $\ze_s$. Since $L$ is special Lagrangian there exists a similar identification of $T_L$ and $T^*L$, so that the graph of $\si_s$ now corresponds to a $3$-submanifold $L_s$ of $X$ asymptotic to $N_s\t\{p\}\t\R$.

Notice that for each $s$ we have chosen a section $\si_s$, so that we need to show the independence of this choice on the problem. What we are interested in are the values of $\om$ and $\Im\Om$ when restricted to the submanifolds $L_s$, so we need only show independence of the choice of $L_s$ on these values. Note that when $L_s$ is asymptotic to $N_s\t\{p\}\t (R,\infty)$, then near infinity $\om|_{L_s}\equiv O(e^{\ga t})$ and $(\Im\Om)|_{L_s}\equiv O(e^{\ga t})$. On the other hand, since $N_s$ is a special Lagrangian, by the long exact sequence below, (3),  $\om|_{L_s}$, $(\Im\Om)|_{L_s}$ are compactly supported; that is, we can take $[\om|_{L_s}]\in H^2_{cs}(L_s,\R)$, $[(\Im\Om)|_{L_s}]\in H^3_{cs}(L_s,\R)$. Moreover, since $N_0$ and $N_s$ are isotopic we can take $[\om|_{L_s}]\in H^2_{cs}(L_0,\R)$ and $[(\Im\Om)|_{L_s}]\in H^3_{cs}(L_0,\R)$ which are independent of the choice of $L_s$; by Lemma \ref{rigidity}, we in fact have that each $L_s$ is asymptotically cylindrical special Lagrangian, asymptotic to $N_s\t\{p\}\t\R$ for some small decay rate.

Consider the linear map $\Up:H^1(N)\to H^2_{cs}(L)$ coming from the long exact cohomological sequence:
\begin{equation}
\begin{split}
0\to H^0(L)\to H^0(N)\to H^1_{cs}(L)\to H^1(L)\to &H^1(N)\\
&\downarrow\\
0\leftarrow H^3_{cs}(L)\leftarrow H^2(N)\leftarrow H^2(L)\leftarrow &H^2_{cs}(L)
\end{split}
\label{les}
\end{equation}
$\Up$ is defined explicitly as follows: let $\eta$ be a closed $1$-form on $N$. Extend $\eta$ to a $1$-form $\tilde{\eta}$ on $L$ asymptotic to $\eta$. Then $\d\tilde{\eta}$ is a closed, compactly-supported $2$-form on $L$. We let $\Up([\eta])=[\d\tilde{\eta}]$. From here, notice that $[\d\tilde{\eta}]=0$ if and only if $\tilde{\eta}$ is a compactly-supported $1$-form, that is, if and only if $\tilde{\eta}$ corresponds to an asymptotically cylindrical special Lagrangian $3$-fold near $L$ in $X$. Thus, we define $\U$ to be an open, connected, simply-connected subset of $\ker\Up$.

Recall from the proof of Theorem \ref{sl1thm} we had the map $\Th:B_{\ep'}(T^*L)\to T_L\subset X$, simply the explicit identification of small sections of $T^*L$ with the tubular neighborhood $T_L$ which is compatible with our asymptotic identifications. Now, we can assume without loss of generality that for $\eta\in L^p_{2+l,\ga}(B_{\ep'}(T^*L)$ and $\si_s$ defined above we have $\eta+\si_s:L\to B_{\ep'}(T^*L)$, that is, a smooth section of the bundle $B_{\ep'}(T^*L)$, so that $\Th\circ(\eta+\si_s):L\to T_L$ makes sense. Thus we can define our deformation map as $$G:\U\t L^p_{2+l,\ga}(B_{\ep'}(T^*L))\to\{2\text{-forms on }L\}\op \{3\text{-forms on }L\}$$ by $$G(s,\eta)=((\Th\circ(\eta+\si_s))^*(-\om),(\Th\circ(\eta+\si_s))^*(\Im\Om)).$$ Note that this map is equivalent to the one defined by $$G(s,\eta)=(\pi_*(-\om|_{\Ga(\eta+\si_s)}),\pi_*((\Im\Om)|_{\Ga(\eta+\si_s)}))$$ where $\pi$ is the projection onto $L$ and $\Ga(\eta+\si_s)$ is the graph of $\Ga(\eta+\si_s)$ in $X$, that is, a $3$-submanifold $\tilde{L}$ of $X$. Finally, notice that $\tilde{L}$ is special Lagrangian if and only if $G(s,\eta)=(0,0)$ so that $G^{-1}(0,0)$ parameterizes the special Lagrangian submanifolds near $L$ in $X$.

We now need to show that $G^{-1}(0,0)$ is smooth, finite-dimensional and locally isomorphic to $\ker((\d+*\d^*)^p_{2+l,\ga})\subset\U\t L^p_{2+l,\ga}(T^*L)$. This is accomplished in much that same way as in the fixed boundary case. We show that $G$ is a smooth map when considered as a map between weighted Sobolev spaces, prove that the image of $G$ lies in the image of the asymptotically cylindrical linear elliptic operator $(\d+*\d^*)^p_{2+l,\ga}$ which will then allow us to use the Implicit Mapping Theorem for Banach Manifolds.

\begin{prop}
$G:\U\t L^p_{2+l,\ga}(B_{\ep'}(T^*L))\to L^p_{1+l,\ga}(\La^2T^*L)\op L^p_{1+l,\ga}(\La^3T^*L)$ is a smooth map of Banach manifolds with linearization at $0$ given by $\d G_{(0,0)}(s,\eta)=(\d(\eta+\si_s),*\d^*(\eta+\si_s))$.
\label{Gmap}
\end{prop}

\begin{proof}
The functional form of $G$ is $$G(s,\eta)|_x=H(s,x,\eta|_x, \na\eta|_x),\text{ }x\in L$$ where $H$ is a smooth function of its arguments. Since $p>3$ and $l\geq 1$ by Sobolev embedding theorem we have $L^p_{2+l,\ga}(T^*L)\hookrightarrow C^1_{\ga}(T^*L)$. General arguments then show that locally $G(s,\eta)$ is $L^p_{1+l}$.

From \cite{SaAJ}, we know that $G(0,\eta)$ is in $L^p_{1+l,\ga}(\La^2T^*L)\op L^p_{1+l,\ga}(\La^3T^*L)$. When $s\neq 0$, then $G(s,0)=(\pi_*(-\om|_{\Ga(\si_s)}), \pi_*((\Im\Om)|_{\Ga(\si_s)}))$. By construction $\Ga(\si_s)$ is asymptotic to $N_s\t\{p\}\t\R$ which is special Lagrangian in $M\t S^1\t\R$ as $N_s$ is a special Lagrangian $2$-fold. Then $\om$ and $\Im\Om$ on $X\setminus K$ can be written as $\om=\ka_I+\d\th\w\d t+O(e^{\al t})$ and $\Im\Om=\Im[(\ka_J+i\ka_K)\w(\d\th+i\d t)]+O(e^{\al t})$ where $\al$ is the rate for $X$ converging to $M\t S^1\t\R$. $\Ga(\si_s)$ is the graph of $\si_s$ which is equal to $N_s\t\{p\}\t(R+1,\infty)$ in $M\t S^1\t(R+1,\infty)$. Since $N_s$ is special Lagrangian, for $t>R+1$, $\om|_{\Ga(\si_s)}=(\ka_I+\d\th\w\d t+O(e^{\al t}))|_{N_s\t\{p\}\t(R+1,\infty)}=0+O(e^{\al t})|_{N_s\t\{p\}\t(R+1,\infty)}$ and $(\Im\Om)|_{\Ga(\si_s)}=(\Im[(\ka_J+i\ka_K)\w(\d\th+i\d t)]+O(e^{\al t}))|_{N_s\t\{p\}\t(R+1,\infty)}=0+O(e^{\al t})|_{N_s\t\{p\}\t(R+1,\infty)}$.

Therefore for $\Ga(\si_s)$ the error terms $\om|_{\Ga(\si_s)}$ and $(\Im\Om)|_{\Ga(\si_s)}$ come from the degree of the asymptotic decay. In particular, as $\al<\ga$ we can assume $\om|_{\Ga(\si_s)}\equiv O(e^{\ga t})$ and $(\Im\Om)|_{\Ga(\si_s)}\equiv O(e^{\ga t})$. We can easily arrange to choose $\si_s$ such that $\pi_*(\om|_{\Ga(\si_s)})\in L^p_{1+l,\ga}(\La^2T^*L)$ and $\pi_*((\Im\Om)|_{\Ga(\si_s)})\in L^p_{1+l,\ga}(\La^3T^*L)$ and that $\Vert \pi_*(\om|_{\Ga(\si_s)}) \Vert_{L^p_{1+l,\ga}}\leq c_1|s|$ and $\Vert \pi_*((\Im\Om)|_{\Ga(\si_s)}) \Vert_{L^p_{1+l,\ga}}\leq c_2|s|$ for some constants $c_1$ and $c_2$. This implies that $G(s,\eta)\in L^p_{1+l,\ga}(\La^2T^*L)\op L^p_{1+l,\ga}(\La^3T^*L)$.

The linearization of $G$ is $\d+*\d^*$ by the same calculation as before since we are still only working locally.
\end{proof}

\begin{prop}
The image of $G$ lies in exact $2$-forms and exact $3$-forms; specifically,
\begin{equation*}
\begin{split}
G(\U\t L^p_{2+l,\ga}(B_{\ep'}(T^*L)))&\subset \d(L^p_{1+l,\ga}(T^*L))\op\d(L^p_{1+l,\ga}(\La^2T^*L))\\
&\subset L^p_{l,\ga}(\La^2T^*L)\op L^p_{l,\ga}(\La^3T^*L).
\end{split}
\end{equation*}
\end{prop}

\begin{proof}
We need only modify slightly the proof from \cite{SaAJ}. Recall that $\om$ and $\Im\Om$ are closed forms, so they determine the de Rham cohomology classes $[\om]$ and $[\Im\Om]$. In particular, $\om|_L\equiv 0$ and $\Im\Om|_L\equiv 0$ since $L$ is special Lagrangian, so $[\om|_L]=0$ and $[\Im\Om|_L]=0$; moreover, since $T_L$, the tubular neighborhood of $L$ from above, retracts onto $L$, $[\om|_{T_L}]=[\om|_L]$ and $[\Im\Om|_{T_L}]=[\Im\Om|_L]$. Thus, there exists $\tau_1\in C^{\infty}(T^*T_L)$ such that $\om|_{T_L}=\d\tau_1$ and $\tau_2\in C^{\infty}(\La^2T^*T_L)$ such that $\Im\Om|_{T_L}=\d\tau_2$. Now, since $\om|_{L}\equiv 0$, we can assume that $\tau_1|_L\equiv 0$; second, because $\om$ and all its derivatives decay at a rate $O(e^{\be t})$ to the translation-invariant $2$-form $\om_0$ on $M\t S^1\t\R$, we can assume that $\tau_1$ and all its derivatives decay at a rate $O(e^{\be t})$ to a translation invariant $1$-form on $T_N\t\{p\}\t\R$. We can make similar assumptions regarding $\tau_2$ based on the properties of $\Im\Om$. Thus, we have $\tau_1\in L^p_{1+l,\ga}(T^*L)$ and $\tau_2\in L^p_{1+l,\ga}(\La^2T^*L)$.

By Proposition \ref{Gmap}, the map $G$ maps $\U\t L^p_{2+l,\ga}(B_{\ep'}(T^*L))\to L^p_{1+l,\ga}(\La^2T^*L)\op L^p_{1+l,\ga}(\La^3T^*L)$. Thus for $(s,\eta)\in \U\t L^p_{2+l,\ga}(T^*L)$ we have:
\begin{equation*}
\begin{split}
G(s,\eta)&=((\Th\circ(\eta+\si_s))^*(-\om),(\Th\circ(\eta+\si_s))^*(\Im\Om))\\
&=((\Th\circ(\eta+\si_s))^*(-\d\tau_1),(\Th\circ(\eta+\si_s))^*(\d\tau_2))\\
&=(\d(\Th\circ(\eta+\si_s))^*(-\tau_1),\d(\Th\circ(\eta+\si_s))^*(\tau_2)).
\end{split}
\end{equation*}
\end{proof}

\begin{prop}
Let $\mathcal{C}$ denote the image of the operator $(\d+*\d^*)^p_{2+l,\ga}$. Then $$G:\U\t L^p_{2+l,\ga}(B_{\ep'}(T^*L))\to\mathcal{C}.$$
\label{image}
\end{prop}

\begin{proof}
Let $(s,\eta)\in\U\t L^p_{2+l,\ga}(B_{\ep'}(T^*L))$. From our previous work \cite[Theorem 3.10]{SaAJ}, $\coker((\d+*\d^*)^p_{2+l,\ga})\cong(\ker((\d^*+\d*)^q_{2+m,-\ga}))^*$ with $\frac{1}{p}+\frac{1}{q}=1$ and $m\geq 1$, so $G(s,\eta)\in\mathcal{C}$ if and only if $$\langle G(s,\eta),(\chi_1,\chi_2)\rangle_{L^2}\equiv0\text{ for all }(\chi_1,\chi_2)\in\ker((\d^*+\d*)^q_{2+m,-\ga}).$$ By the previous proposition $G(s,\eta)=(\d\tau_1,\d\tau_2)$ for some $\tau_1\in L^p_{1+l,\ga}(T^*L)$ and $\tau_2\in L^p_{1+l,\ga}(\La^2T^*L)$; then \cite[Theorem 3.10]{SaAJ} shows $\chi_1$ and $\chi_2$ are coclosed $2$- and $3$-forms respectively which yields the following:
\begin{equation*}
\begin{split}
\langle G(s,\eta),(\chi_1,\chi_2)\rangle_{L^2}&=\langle \d\tau_1,\chi_1\rangle_{L^2}+\langle\d\tau_2,\chi_2\rangle_{L^2}\\
&=\langle \tau_1,\d^*\chi_1\rangle_{L^2}+\langle\tau_2,\d^*\chi_2\rangle_{L^2}=0.
\end{split}
\end{equation*}
\end{proof}

We are now able to invoke to the Implicit Mapping Theorem for Banach Spaces to conclude that $G^{-1}(0,0)$ is smooth, finite-dimensional and locally isomorphic to $\ker((\d+*\d^*)^p_{2+l,\ga})\subset\U\t L^p_{2+l,\ga}(T^*L)$ as desired. The final step is to calculate the dimension of the moduli space which we do by using Fredholm index arguments. (The rest of the proof of Theorem \ref{mslthm1} will then follow exactly as the proof of Theorem \ref{sl1thm}, so we will not repeat it here refering the interested reader instead to our previous work \cite[Lemmas 4.4-4.5, Proposition 4.6]{SaAJ}.)

\begin{prop}
Let $\U\subset H^1(N,\R)$ be a subspace of special Lagrangian deformations of the boundary $N$. Also, let $\d G_{(0,0)}(s,\eta)$ represent the linearzation of the deformation map $G$ at $0$ with moving boundary and $\d G^f_{(0,0)}(\eta)$ represent the linearzation of the deformation map $G^f$ at $0$ with fixed boundary. Then $$\Ind(\d G_{(0,0)})=\dim\U+\Ind(\d G^f_{(0,0)})$$ where $\Ind$ denotes the index of a map.
\end{prop}

\begin{proof}
At $s=0$ $$G(0,\eta)=G^f(\eta)=(\pi_*(-\om|_{\Ga_{\eta}}), \pi_*((\Im\Om)|_{\Ga_{\eta}}))$$ with linearization at $(0,0)$ $$\d G_{(0,0)}(0,\eta)=\d G^f_{(0,0)}(\eta).$$ Since $\d G_{(0,0)}$ is linear, we have $$\d G_{(0,0)}(s,\eta)=\d G_{(0,0)}(s,0)+\d G^f_{(0,0)}(0,\eta)$$ where $\d G_{(0,0)}(s,0)$ is finite-dimensional, $s\in T_0\U\cong\R^d$, $d=\dim\ker\Up$. Then
$$\Ind(\d G_{(0,0)}:\U\t L^p_{2+l,\ga}(T^*L)\to L^p_{1+l,\ga}(\La^2T^*L)\op L^p_{1+l,\ga}(\La^3T^*L))$$
$$=\Ind(\d G^f_{(0,0)}:\U\t L^p_{2+l,\ga}(T^*L)\to L^p_{1+l,\ga}(\La^2T^*L)\op L^p_{1+l,\ga}(\La^3T^*L))$$
$$=\dim\U+\Ind(\d G^f_{(0,0)}).$$
\end{proof}

\begin{prop}
The dimension of $\M^{\ga}_L$, the moduli space of special Lagrangian deformations of an asymptotically cylindrical special Lagrangian submanifold $L$ asymptotic to $N_s\t\{p\}\t(R,\infty)$, $s\in\U$, with decay rate $\ga$ is $$\dim\M^{\ga}_L=\dim V+b^2(L)-b^1(L)+b^0(L)-b^2(N)+b^1(N)$$.
\end{prop}

\begin{proof}
Let $b^k(N)$, $b^k_{cs}(L)$ and $b^k(L)$ be the corresponding Betti numbers. We know from the previous proposition that $$\Ind(\d G_{(0,0)})=\dim\U+\Ind(\d G^f_{(0,0)}).$$ In particular, $$\dim(\ker\d G_{(0,0)})=\dim(\ker\Up)+\dim(\ker\d G^f_{(0,0)},$$ that is, the dimension of the moduli space for the moving boundary case is the sum of the dimension of the moduli space for the fixed boundary case and the dimension of the kernel of $\Up$. Taking alternating sums of dimensions in the long exact sequence (\ref{les}) shows that the dimension of the kernel of $\Up$ is $$\dim(\ker\Up)=b^1(N)-b^2_{cs}(L)+b^2(L)-b^2(N)+b^3_{cs}(L)$$ $$=b^2(L)-b^1(L)+b^0(L)-b^2(N)+b^1(N).$$
\end{proof}

\section{Theorem \ref{lagimm}}
Before we prove Theorem \ref{lagimm}, we need just a bit of background. Let $(X,\om,g,\Om)$ be an asymptotically cylindrical Calabi-Yau $3$-fold asymptotic to $M\t S^1\t(R,\infty)$, and $L$ an asymptotically cylindrical special Lagrangian submanifold of $X$ asymptotic to $N\t\{p\}\t(R',\infty)$. Let $E$ be a fixed rank one vector bundle over $L$. A connection $D_E$ on $L$ has finite energy if $$\int_L|F_E|^2\d V<\infty$$ where $F_E$ is the curvature of the connection $D_E$ and $\d V$ is the volume form with respect to the metric.

\begin{dfn}
Using the notation above, the pair $(L,D_E)$ is called a special Lagrangian cycle if $L$  is a special Lagrangian submanifold of $X$ asymptotic to $N\t\{p\}\t(R',\infty)$ and $D_E$ is a unitary flat connection over $L$ with finite energy.

The pair $(N,D_E')$ is a special Lagrangian cycle if $D_E'$ is a unitary flat connection over $N$ induced from $D_E$.
\end{dfn}

Let $\mathcal{M}^{SLag}(M)$ be the moduli space of special Lagrangian cycles in $M$. In \cite{Hi}, Hitchin proves the following theorem:

\begin{thm}
The tangent space to $\mathcal{M}^{SLag}(M)$ is naturally identified with the space $H^1(N,\R)\t H^0(N,ad(E'))$. For line bundles over $N$, the cup product $\cup:H^1(N,\R)\t H^0(N,\R)\to\R$ induces a symplectic structure on $\mathcal{M}^{SLag}(M)$.
\end{thm}

We now prove Theorem \ref{lagimm}.

\begin{proof}[Proof of Theorem \ref{lagimm}]
Let $\de:H^0(N)\to H^1(L,N)$, $j^*:H^1(L,N)\to H^1(L)$, $i^*:H^1(L)\to H^1(N)$ be the canonical maps. They give dual maps $i_*:H_1(N)\to H_1(L)$, $j_*:H_1(L)\to H_1(L,N)$ and $\partial:H_1(L,N)\to H_0(N)$. Then we have the following long exact sequence:
\begin{equation*}
\begin{array}{ccccccccccccc}
    \cdots & \to & H^0(N) & \to & H^1(L,N) & \to & H^1(L) & \to & H^1(N) & \to & H^2(L,N) & \to & \cdots\\
    & & \downarrow\cong & & \downarrow\cong & & \downarrow\cong & & \downarrow\cong & & \downarrow\cong & & \\
    \cdots & \to & H_1(N) & \to & H_1(L) & \to & H_1(L,N) & \to & H_0(N) & \to & H_0(L) & \to & \cdots\\
\end{array}
\end{equation*}
Note that the classes coming from the boundary $N$ should have zero self-intersection, so we can rewrite the long exact sequences above starting from $0$ instead of $H^0(N)$ and $H_1(N)$. (Note that we are not claiming that $H^0(N)$ and $H_1(N)$ are trivial!)

From the sequences we get $$\ker\de\cong\ker i_*\cong\im i^*\cong\im\partial\cong \frac{H_1(L,N)}{\ker\partial}= \frac{H_1(L,N)}{\im j_*}$$ which implies
\begin{equation}
\dim(\ker\de)=\dim(H_1(L,N))-\dim(\im j_*)
\label{dim}
\end{equation}
Note that the only classes that survive in $H_1(L)\to H_1(L,N)$, i. e., do not go to zero, have self-intersection zero. Thus, one can indentify $\im j_*$ with $H_1(L)$. This implies that there is a short exact sequence $$0\to\im j_*\to H_1(L,N)\to\im\partial\to 0,$$ or equivalently $$0\to H_1(L)\to H_1(L,N)\to\ker\de\to 0.$$ Then as a consequence of (\ref{dim}) and Theorem \ref{mslthm1} we conclude that $H^1(L,N)=V\op\ker\de$ parameterizes the deformations of $L$ with moving boundary $\partial L$; from our previous work, it follows that $H^1(L)\cong V$ parameterizes the deformations of $L$ with fixed boundary, and McLean showed that $H^1(N)$ gives the special Lagrangian deformations of $N$.

Note that the linearization of the boundary value map $\mathcal{M}^{SLag}(X)\to\mathcal{M}^{SLag}(M)$ in this theorem is given by $i^*$ and for the connection part $\be:H^0(L)\to H^0(N)$. It is straightforward that $\im i^*\op\im\be$ is a subspace of $H^1(N)\op H^0(N)$. By definition of the cup product, the symplectic structure reduces to $0$ on $\im i^*\op\im\be$ and by Poincar\'e duality, $\dim(\im i^*\op\im\be)=\frac{1}{2}\dim(H^1(N)\op H^0(N))$. Thus, we conclude that $\im i^*\op\im\be$ is a Lagrangian subspace of $H^1(N)\op H^0(N)$ with the symplectic structure defined above and conclude the proof of this theorem.
\end{proof}

\end{document}